\documentclass[a4paper,10pt]{article}
\usepackage{mathrsfs}
\usepackage{latexsym}
\usepackage{amsmath,amssymb}
\usepackage{amsthm}
\usepackage{amsfonts}
\usepackage[usenames]{color}
\usepackage{amssymb}
\usepackage{graphicx}
\usepackage{amsmath}
\usepackage{amsfonts}
\usepackage{amsthm}
\usepackage{mathrsfs}
\usepackage{dsfont}
\usepackage{indentfirst}

\newtheoremstyle{mythm}{1.5ex plus 1ex minus .2ex}{1.5ex plus 1ex
minus .2ex}{\kai}{\parindent}{\song\bfseries}{}{1em}{}
\numberwithin{equation}{section}

\newtheorem{theorem}{Theorem}[section]
\newtheorem{lemma}{Lemma}[section]

\newtheorem{remark}{Remark} [section]

\allowdisplaybreaks[4]
\begin{document}

\title{{\textbf{A Liouville theorem of degenerate elliptic equation and its application}}}
\author{Genggeng Huang\footnote{genggenghuang@fudan.edu.cn}}
\date{}
\maketitle
\begin{center}
(School of Mathematics Sciences,  LMNS,  Fudan University, Shanghai,
China,  200433)
\end{center}

\begin{abstract}
In this paper, we  apply the moving plane method to some degenerate elliptic equations to get a Liouville type theorem.
As an application, we derive the a priori bounds for positive solutions of  some semi-linear
degenerate elliptic equations.

\end{abstract}

\section{Introduction}\setcounter{section}{1} \setcounter{equation}{0}
\setcounter{theorem}{0}\setcounter{lemma}{0}
\label{intro}
In the present paper, we  study the nonnegative solutions $u(x,y)$ of the
following equation with a constant $a>1$
\begin{equation}\begin{cases}\label{001}
yu_{yy}+au_y+\Delta_x u+u^\alpha=0\text{ in } R^{n+1}_+,n\geq 1,\\
u(x,y)\geq 0,u(x,y)\in C^2(\overline{R^{n+1}_+}),1<\alpha\leq
\frac{n+2a+2}{n+2a-2}.
\end{cases}\end{equation} Notice that no boundary
condition is  imposed on $y=0$ which is the characteristic of
\eqref{001}. As far as I know, I haven't seen any Liouville type theorem concerning \eqref{001}.  Set $x_{n+1}=2\sqrt y$. \eqref{001} changes to \begin{equation}\label{hgg2}
\Delta_{n+1} u+\frac{2a-1}{x_{n+1}}\partial_{n+1}u+u^{\alpha}=0,\text{ in } R^{n+1}_+
\end{equation}  If $a=\frac k2,k\in N^+$, then we have \begin{equation}\begin{cases}
\label{hgg1}
\Delta_{x',\xi}v+v^{\alpha}=0,\text{ in } R^{n+k}\backslash \{\xi_1=0,..,\xi_k=0\}\\
v\geq 0,v\in C(R^{n+k}),1<\alpha\leq
\frac{n+k+2}{n+k-2}.
\end{cases}\end{equation} with $v(x_1,...,x_n,\xi_1,..,\xi_{k})=u(x_1,...,x_n,\sqrt{\xi_1^2+...+\xi^2_k})$. For \eqref{hgg1}, it is quite similar to the following problem except for a codimension hyperplane,\begin{equation}\begin{cases}\label{002}
\Delta u+u^q= \text{ in } R^n, n>2,\\ u(x)\geq 0,\  u(x)\in
C^{2}(R^n),1<q\leq \frac{n+2}{n-2}.
\end{cases}
\end{equation}
The above problem \eqref{002} was  investigated in \cite{GS} and \cite{CGS}. For the subcritical case $1<q<\frac{n+2}{n-2}$, the
 only
nonnegative solution $u(x)$ is  trivial. Whereas for the
critical case $q=\frac{n+2}{n-2}$, $u(x)$ is in  a two parameter family of
functions as
\begin{equation}\label{004}
u_{t,x_0}(x)=\left(\frac{t\sqrt{n(n-2)}}{t^2+|x-x_0|^2}\right)^{\frac{n-2}{2}}.
\end{equation}
There are many extended results of problem \eqref{002} which mainly concern on the higher order case, \begin{equation}\begin{cases}
(-\Delta)^p u=u^q \text{ in } R^n, n>2p,\\ u(x)\geq 0,\  u(x)\in
C^{2p}(R^n),1<q\leq \frac{n+2p}{n-2p}.
\end{cases}
\end{equation}
For the case $p=2$, Lin and Xu got the similar results respectively
in \cite{L} and \cite{X}.  Wei-Xu extended the results to the case
$2\leq 2p\leq n,p\in \mathbb{Z}$ in \cite{WX}. Chen-Li-Ou and Li
 proved the results for the most general case $0<p<\frac n2$ by
 the integral form of the moving plane method(moving sphere
method) respectively in \cite{CLO} and \cite{LYY}. Applying
Chen-Li-Ou's method to systems as in \cite{CL} and \cite{LM}, one
can also get the similar conclusions. Chang and Yang in \cite{CY} also extended this results to manifolds.
\par The main method used in solving problem \eqref{002} is the moving plane method which was first
proposed by Alexandrov \cite{A} and developed by  Serrin \cite{S}, Gidas, Ni and
Nirenberg \cite{GNN,GNN2}. Now moving plane method has been widely used
in study of the symmetry of the positive solutions of many
elliptic partial differential equations and systems. The key point of using the moving plane method in \eqref{002} is the conformal invariant property and the rotation invariant property of \eqref{002}.

 \par In our case, we also use the moving plane method and the conformal invariant property. To do so, we must establish some new maximal principles and overcome the difficult that \eqref{hgg2} is not rotation invariant.

  In this paper,
 we obtain the following results for \eqref{001}.
\begin{theorem}\label{thm1} Let $u(x,y)\geq 0$ be a
solution to \eqref{001} with $a>1$. Then
\begin{itemize}

\item[(1)] for $1<\alpha<\frac{n+2a+2}{n+2a-2},\ u(x,y)\equiv 0$,
\item[(2)] for $\alpha=\frac{n+2a+2}{n+2a-2}$,
$\displaystyle u_{t,x_0}(x,y)=\left(\frac{t\sqrt{(n+2a)(n+2a-2)}}{t^2+4y+|x-x_0|^2}\right)^{
\frac{n+2a-2}{2}}$,
\end{itemize}for some $x_0\in R^n$ and $t\geq 0$.
\end{theorem}
It is easy to see that  for $a=\frac k2,k\in N^+$, Theorem \ref{thm1} is exactly the result of \eqref{002}  in $R^{n+k}$. For general $a>1$, we may consider Theorem \ref{thm1} as the extension of the results of \eqref{002} to $R^{n+2a}$ with real dimension $n+2a$.
 \par As an application of
Theorem \ref{thm1}, we also derive a priori bounds for positive solutions of some
semi-linear degenerate elliptic  equations which arising from the study of geometry,
\begin{equation}\label{601}
a^{ij}(x)\partial_{ij}u+b^i(x)\partial_i u +f(x,u)=0,\text{ in
}\Omega\subset\subset R^2.
\end{equation}  Let $\phi\in C^2(\mathcal{N}(\partial \Omega))$ be the defining function of $\partial
\Omega$, namely,
\begin{equation}\label{602}\phi|_{\partial\Omega}=0,
\nabla\phi|_{\partial \Omega}\neq 0,\phi>0\text{ in  } \Omega\cap\mathcal{N}(\partial \Omega)
\end{equation}where $\mathcal{N}(\partial \Omega)$ is a neighborhood of $\partial \Omega$. Also, suppose that \begin{equation}\label{666}(a^{ij})>0 \text{ in } \Omega,
a^{ij}(x)\partial_i\phi\partial_j\phi=0,\nabla(a^{ij}\partial_i\phi\partial_j\phi)\neq 0
\text{ on }\partial \Omega
\in C^{2}
\end{equation} and that near $\partial\Omega$ for the  eigenvalues of $(a^{ij})$  $\lambda_1$ and $\lambda_2$, there hold, for some constant $c_0$,
\begin{equation}\label{603} \lambda_1\geq
c_0>0,\lambda_{2}=m(x)\phi, 0<m(x)\in C(\bar \Omega).
\end{equation}\begin{theorem}\label{thm501}Let \eqref{602}, \eqref{666} and \eqref{603} be
fulfilled.
Suppose that $0<u\in C^2(\Omega)\cap L^{\infty}(\Omega)$ solves
\eqref{601} and that  $a^{ij}\in
C^{2}(\bar{\Omega}),b^i\in C^{1}(\bar{\Omega})$ and $f(x,t)\in C(\bar{\Omega}\times[0,\infty))$
\begin{equation}
\label{604}\underset{t\rightarrow
\infty}{\lim}\frac{f(x,t)}{t^\alpha}=h(x), \text{ uniformly for some
} 1<\alpha<\frac{3+2a}{2a-1},
\end{equation} where
$0<h(x)\in C(\bar{\Omega})$,\begin{equation}\label{a101}a=\displaystyle \underset{\partial\Omega}{\sup}\frac{b^i\partial_i\phi-\partial_ja^{ij}\partial_{i}
\phi}{\partial_ka^{ij}\partial_i\phi\partial_j\phi\phi^k}+1,
b\triangleq\underset{\partial\Omega}{\inf}\displaystyle\frac{b^i\partial_i\phi-\partial_ja^{ij}\partial_{i}
\phi}{\partial_{k}a^{ij}\partial_i
\phi\partial_j\phi\phi^k}>1,\phi^k=\frac{\partial_k\phi}{|\nabla \phi|^2}.
\end{equation}
  Then it follows that \begin{equation}\label{888}|u|_{L^\infty}\leq
C.\end{equation}

\end{theorem}

\begin{remark}
Define $$g(x)=\frac{b^i\partial_i\phi-\partial_ja^{ij}\partial_{i}
\phi}{\partial_ka^{ij}\partial_i\phi\partial_j\phi\phi^k}\text{ on }\partial \Omega \text{ where } a^{ij}\partial_i\phi\partial_j\phi=0.$$ The invariance of $g(x)$ is proved in \cite{H}. The numerator of $g(x)$ is the well-know Fichera number. The concept of Fichera number is very important when we deal with degenerate elliptic problems with boundary characteristic degenerate. It indicates whether we should impose boundary condition in such case. This fact was first observed by M.V.Keldy\text{$\check{s}$} in \cite{Kel} and developed by Fichera in \cite{Fic1,Fic2}. The Fichera number also affects the regularities of the solutions up to the boundary, see \cite{H}. For more details of Fichera number, refer to \cite{O}.
\end{remark}
\begin{remark}
It might be  hard to understand that the nonlinearity of $f(x,u)$ should be related to $g(x)$. We can take equation \eqref{001} for instance to explain why this happens. In this situation, $\phi=y,f=u^{\alpha}$. It is easy to see $g(x,0)=a-1$ by a direct computation. Theorem \ref{thm1} tells us that the existence of non-trivial positive solution depends on the nonlinear power $\alpha$ which is involved in $a$. When we use blow up method to get a priori estimates of  \eqref{601}, one of the limit cases is \eqref{001} as the blow up point approaching the boundary. It is nature that the the nonlinearity of $f(x,u)$ should be related to $g(x)$ if we want to get the a priori bounds.
\end{remark}
The proof of Theorem \ref{thm501}  mainly follows the blow up method used in \cite{GS2}. The mainly difficulty we encounter is the case when the blow-up point approach to the boundary. This case becomes complicated with the degeneracy of the equation on the boundary and without boundary condition. We  should take a suitable transformation of coordinates to make the limit equation exists and establish some regularities estimates up to the boundary to guarantee the point-wise convergence.

The present paper is organized as follows. In Section 2, we
establish some lemmas which are similar to Lemma 2.3 and Lemma 2.4 in \cite{CGS},
and  necessary for utility of the moving plane method. In
Section 3, we shall use the moving plane method to prove Theorem
\ref{thm1}. In Section 4,  as an application of Theorem \ref{thm1}, we derive a
priori bounds for positive solutions of  some semi-linear degenerate
elliptic equations.

\section{Preliminary Results}
\label{sec:1}
In this section we
collect some  preliminary results which will be needed for our later
analysis.
\par Suppose that $u$ solves \eqref{001}.
Set $x_{n+1}=2\sqrt y$ and
$\bar{u}(x_1,...,x_{n+1})=u(x_1,...,x_n,\frac {x_{n+1}^2}4)$. Then
\eqref{001} is reduced to
\begin{equation}\label{103}
\Delta_{n+1} \bar{u}+\frac{2a-1}{x_{n+1}}
\partial_{n+1}\bar{u}+\bar{u}^\alpha=0 \text{ in } R^{n+1}_+.
\end{equation} Noting that $u\in C^2(\overline{R^{n+1}_+})$, we
must have
$$\partial_{n+1}\bar{u}=\frac{x_{n+1}}2u_y\Rightarrow \partial_{n+1}
\bar{u}(x',0)=0\text{ where }x'=(x_1,...,x_n).$$  This allows us to
extend $\bar{u}$ to the lower half-space by
$$\bar{u}(x',x_{n+1})=\bar{u}(x',-x_{n+1})\text{ for } x_{n+1}<0,$$
such that $\bar{u}(x)\in C^2(R^{n+1})$ with
$\partial_{n+1}\bar{u}(x',0)=0$.\par Consider the following
elliptic operator
$$L(u)=\displaystyle \sum_{i=1}^{n+1} a_{ij}(x)\partial_{ij}u+\sum_{i=1}^n
b_i(x) \partial_iu+\frac{a(x)}{x_{n+1}}\partial_{n+1}u.$$ All the
coefficients $a_{ij}(x),b_i(x),a(x)\in C(R^{n+1})$, $a(x)\geq  0$
and $(a_{ij})$ is a positive definite matrix. Then we shall have
the following two lemmas.
\begin{lemma}\label{lem001} Suppose that $u\in
C^2(B_1)\cap C(\bar{B}_1)$ with $\partial_{n+1}u(x',0)=0$ satisfies
that\begin{equation*} -L(u)\geq 0 \text{ in } B_1.
\end{equation*} Then either $u$ is a constant or $u$ can not attain its minimum in
$B_1$.
\end{lemma}
\begin{proof}It suffices to prove the second case. Without loss of
generality, we may assume that $u$ attains its minimum at the
origin.

Denote $$B_r(P)=\{x\in B_1\big|x'^2+(x_{n+1}-r)^2\leq r^2\}, r<\frac
12, \text{ where }P=(0',r)\in R^{n+1}_+.$$ $B_{\frac r2}(P)$ has the
same center as $B_{r}(P)$ but half radius. Set
$\Sigma=B_r(P)\backslash B_{\frac r2}(P)$. We consider
$$h(x)=1-e^{-\beta[x'^2+(x_{n+1}-r)^2-r^2]} \text{ in }  \Sigma.$$
 Then we have
\begin{eqnarray}\label{101}
L(h)&=&\displaystyle \sum_{i=1}^{n+1} a_{ij}(x)\partial_{ij}
h+\sum_{i=1}^n
b_i(x) \partial_i h+\frac{a(x)}{x_{n+1}}\partial_{n+1} h\nonumber\\
&\leq &e^{-\beta[x'^2+(x_{n+1}-r)^2-r^2]}\left\{
-c_1\beta^2(x'^2+(x_{n+1}-r)^2)+c_2\beta(1+r)-\frac
{2a(x)\beta r}{x_{n+1}}\right\}\nonumber\\
&\leq &-c_0 \text{ in } \Sigma, \    \   \text{ since } a(x)\geq 0.
\end{eqnarray}
for some positive constant $c_0>0$ if we take $\beta$ large enough.
 Now let $v=u+\epsilon h$, then $Lv<0$. This implies that $v$ must
attain its minimum on the boundary of $\Sigma$. Consider $v$ on
the boundary of $\Sigma$,

(i) on $\partial B_r(P)$,  noting that $h|_{\partial B_r(P)}=0$,
$u|_{\partial B_r(P)}\geq u(0)$, then we have $$u+\epsilon
h|_{\partial B_r(P)}\geq u(0).$$

(ii) on $\partial B_{\frac r2}(P)$, there exists $\delta>0$ such
that $u|_{\partial B_{\frac r2}(P)}\geq u(0)+\delta$. Thus we can
choose $\epsilon$ small enough such that$$u+\epsilon h|_{\partial
B_{\frac r2}(P)}\geq u(0)+\frac{\delta}{2}.$$ This means $u+\epsilon
h\geq u(0)$ in $\Sigma$, which implies
$$
\partial_{n+1}u(0)\geq -\epsilon \partial_{n+1}h(0)>0.$$ This contradicts to
$\partial_{n+1}u(x',0)=0$. This ends the proof of the present lemma.
\end{proof}
\begin{lemma}\label{lem002}
If $u(x)\in C^2(B_1)\cap C^1(\bar{B}_1)$ with
$\partial_{n+1}u(x',0)=0$ satisfies that\begin{equation} -L(u)\geq 0
\text{ in } B_1.
\end{equation}
If $u$ attains its minimum at $x^0\in \partial B_1$, then either
$u\equiv const$ or $$-\frac{\partial u}{\partial n}|_{x=x^0}>0,\ n
\text{ is the outward normal to $\partial B_1$ at } x^0.$$
\end{lemma}
\begin{proof}

Assume that $u$ is not a constant. By Lemma \ref{lem001}, $u(x)$ can
not attain its minimum in $B_1$. If $u$ attains the minimum at
$x^0\in
\partial B_1\backslash\{x_{n+1}=0\}$, it is the immediate consequence
of the  standard Hopf' lemma. Without loss of any generality, we
assume that $v$ attains its minimum at $x^0=(1,0,...,0)$. Set
$$h(x)=e^{-\beta x^2}-e^{-\beta}\ \   \ \ \ \text{ in } B_1\backslash B_{\frac 12}.$$
Then $h(x)\geq 0$, $h|_{\partial B_1}=0$ and
\begin{eqnarray*}L(h)&=&\exp\{-\beta x^2\}
\{4\beta^2a_{ij}x_ix_j -2\beta
a_{ii}-2b_i\beta x_i
-2a(x)\beta \}\\&\geq
&\exp\{-\beta x^2\}\{c_0\beta^2 x^2
-c_1\beta |x|-c_2\}>0,\end{eqnarray*} if we choose
$\beta$ large enough. Choosing $\epsilon>0$ small enough, one can
get
$$ L(v-\epsilon h)<0\text{ in } B_1\backslash B_{\frac
12},\   u-\epsilon h\geq u(x^0)\text{ on }
\partial B_1\cup \partial B_{\frac 12}.$$ An application of  the maximum
principle to $u-\epsilon h$ yields
$$\frac{\partial u}{\partial x_1}|_{x=x^0}
\leq \epsilon\frac{\partial h}{\partial x_1}|_{x=x^0}<0.$$ This completes
the proof of the present lemma.
\end{proof}
Turn back to \eqref{103} and  consider
$$v(x)=\frac{1}{|x|^{n+2a-2}}\bar{u}(\frac{x}{|x|^2}).$$ By a direct computation
with $\tau=n+2a+2-\alpha(n+2a-2)$,
 $v$ solves,  \begin{equation}\label{104} \Delta_{n+1}
v+\frac{2a-1}{x_{n+1}}\partial_{n+1}v+|x|^{-\tau}v^\alpha=0 \text{
in } R^{n+1}\backslash\{0\},\    \partial_{n+1}v(x',0)=0\text{ for }x'\neq 0.
\end{equation}
From the definition of $v$, we will have the following asymptotic
behavior at $\infty$
\begin{equation}\begin{cases}\label{105}
v(x)=\displaystyle\frac{a_0}{|x|^{n+2a-2}}+\displaystyle\sum_{i=1}^{n+1}\frac{a_ix_i}{|x|^{n+2a}}+O(\frac
1{|x|^{n+2a}}),\\
\partial_iv(x)=\displaystyle-\frac{(n+2a-2)a_0x_i}{|x|^{n+2a}}+O(\frac{1}{|x|^{n+2a}}),\\
\partial_{ij}v(x)=\displaystyle O(\frac{1}{|x|^{n+2a}}), \text{ with }a_0>0.
\end{cases}\end{equation}
Next we generalize the important lemmas which are essential for the
application of moving plane method to \eqref{002} in \cite{CGS} to
the equation \eqref{001} studied in the present paper. Denote
$$\Sigma_{\lambda}=\{x\in R^{n+1}|x_1<\lambda\},
x^\lambda=(2\lambda-x_1,x_2,...,x_{n+1}).$$ Then there hold the
following lemmas\begin{lemma}\label{lem003} Let $v$ be a function in
a neighborhood of infinity satisfying the asymptotic expansion
\eqref{105}. Then there exist two positive constants $R,\lambda_1$
such that, if $\lambda\geq \lambda_1$,
$$v(x)>v(x^\lambda),\ \text{ for } x\in \Sigma_{\lambda},|x|>R,\lambda\geq \lambda_1.$$
\end{lemma}
\begin{proof}In view of \eqref{105},
\begin{eqnarray}v(x)-v(x^\lambda)&=&a_0(\frac 1{|x|^{n+2a-2}}-\frac
1{|x^\lambda|^{n+2a-2}})+\sum_{j=1}^{n+1}a_jx_j(\frac
1{|x|^{n+2a}}-\frac
1{|x^\lambda|^{n+2a}})\nonumber\\&+&\frac{2a_1(x_1-\lambda)}{|x^\lambda|^{n+2a}}
+O(\frac{1}{|x|^{n+2a}}),
\end{eqnarray}

(1) if $|x^\lambda|\geq 2|x|$, there holds $$v(x)-v(x^\lambda)\geq
\frac{c_0}{|x|^{n+2a-2}}-\frac{c_1}{|x|^{n+2a-1}}>0 \text{ if }
|x|>R, \text{ for sufficiently large } R,$$

(2) if $|x^\lambda|<2|x|$, there holds $$a_0(\frac
1{|x|^{n+2a-2}}-\frac 1{|x^\lambda|^{n+2a-2}})\geq
\frac{a_0}{|x|^{n+2a-3}}(\frac 1{|x|}-\frac 1{|x^\lambda|})\geq
\frac{c_0(|x^\lambda|-|x|)}{|x|^{n+2a-1}}.$$ Hence
$$v(x)-v(x^\lambda)\geq \frac{c_1(|x^\lambda|-|x|)}{|x|^{n+2a-1}}
-\frac{c_2}{|x|^{n+2a}}.$$ \par If
$|x^\lambda|-|x|>\frac{c_2}{c_1}\frac 1{|x|}$, it follows that
$v(x)>v(x^\lambda)$. \par If $|x^\lambda|-|x|\leq
\frac{c_2}{c_1}\frac 1{|x|}$,  this implies $x_1\geq \frac
\lambda2$. By asymptotic expansion of $v$ at $\infty$, $v_{x_1}<0$
if we choose $\lambda$ large enough, thus $v(x)-v(x^\lambda)>0$.
\end{proof}
\begin{lemma}\label{lem004}
Suppose that $v$ satisfies the assumption of Lemma \ref{lem003} and
\begin{equation}\begin{cases}\Delta_{n+1}(v(x)-v(x^{\lambda_0}))+
\frac{2a-1}{x_{n+1}}\partial_{n+1}(v(x)-v(x^{\lambda_0}))\leq 0
\text{ in }
\Sigma_{\lambda_0}\cap \{|x|>R\},\\
v(x)>v(x^{\lambda_0}),\ \forall x\in
\Sigma_{\lambda_0}\cap\{|x|>R\},\end{cases}\end{equation} with
$\partial_{n+1}v(x',0)=0$ and $a\geq \frac 12$. Then
there exist two constants $\epsilon$ and $S>0$ such that
\begin{itemize}
\item[(i)] $v_{x_1}<0$ in $|x_1-\lambda_0|<\epsilon$ and $|x|>S$,

\item[(ii)] $v(x)>v(x^\lambda)$ in $x_1\leq
\lambda_0-\frac{\epsilon}2\leq \lambda$ and $|x|>S$ for all $x\in
\Sigma_{\lambda}$ with
$|\lambda-\lambda_0|\leq o(\epsilon)$.\end{itemize}
\end{lemma}
\begin{proof}
With $w(x)=v(x)-v(x^{\lambda_0})$,  it is easy to see
$w(x)|_{x_1=\lambda_0}=0$ and that
\begin{equation}\begin{cases}
\Delta w+\frac{2a-1}{x_{n+1}}\partial_{n+1}w\leq 0,\    \  \text{in
}\Sigma_{\lambda_0}\cap\{|x|>R\},\\
w(x)>0,\partial_{n+1}w(x',0)=0\text{ in
}\Sigma_{\lambda_0}\cap\{|x|>R\}.

\end{cases}\end{equation}

We claim: there exists $\delta>0$ so small that
\begin{equation}\label{111}w(x)>\frac{\delta(\lambda_0-x_1)}{|x-\lambda_0 \mathbf{e}_1|^{n+2a}} \ \
\ \   \   \   \text{ in }
\Sigma_{\lambda_0}\cap\{|x|>R+1\}.\end{equation} Here $\mathbf{e}_i$ means the
vector whose i-th coordinate is 1 and others are 0. By Lemma
\ref{lem002}, we see that $w_{x_1}<0$ on
$\{|x|=R+1\}\cap\{x_1=\lambda_0\}$ which implies for some small
$\bar{\delta},k_0>0$,
$$w(x)>\frac{k_0(\lambda_0-x_1)}{|x-\lambda_0 \mathbf{e}_1|^{n+2a}}\text{ on }
\{|x|=R+1\}\cap \{\lambda_0-\bar{\delta}\leq
x_1 < \lambda_0\}.$$
 We shall notice that $w(x)\geq
c_0>0$ on  $\{|x|=R+1\}\cap \{x_1\leq \lambda_0-\bar{\delta}\}$,
thus
$$w(x)>\frac{\delta(\lambda_0-x_1)}{|x-\lambda_0 \mathbf{e}_1|^{n+2a}}
\text{ on } \{|x|=R+1\}\cap\{x_1<\lambda_0\}.$$ Denote
$$h(x)=\frac{\lambda_0-x_1}{|x-\lambda_0\mathbf{e}_1|^{n+2a}}.$$ It is easy to
see that $$\Delta_{n+1}
h+\frac{2a-1}{x_{n+1}}\partial_{n+1}h=0\text{ in }
R^{n+1}\backslash\{\lambda_0\mathbf{e}_1\}.$$ Therefore \eqref{111} is proved
by the maximum principle. In particular,
$$w_{x_1}(\lambda_0,x'')=2v_{x_1}(\lambda_0,x'')<-\delta/|x''|^{n+2a},\text{ where }
x''=(x_2,x_3,...,x_{n+1}).$$ Combining this with the asymptotic
expansion yields
\begin{eqnarray*}
v_{x_1}(\lambda_0+h,x'')&\leq &
v_{x_1}(\lambda_0,x'')+\frac{C|h|}{|x|^{n+2a}}\\
&\leq &-\frac 12\delta\frac
1{|x''|^{n+2a}}+\frac{C|h|}{|x|^{n+2a}}\\&\leq& -\frac 14\delta\frac
1{|x|^{n+2a}},
\end{eqnarray*}as $|h|<\frac{\delta}{4C}$ and $|x|$ is large. This proves
the first part of the present lemma. \par  As for the second part,
from the asymptotic expansion and the result of the first part, it
follows that
$$v(2\lambda_0-x_1,x'')-v(2\lambda-x_1,x'')\geq
\frac{-c(\lambda_0-\lambda)}{|x-\lambda_0e_1|^{n+2a}}(\lambda_0-x_1+c),$$
as $x_1<\lambda_0$ and $|x|$ is large. Hence, \begin{eqnarray*}
&&v(x_1,x'')-v(2\lambda-x_1,x'')\\&=&(v(x_1,x'')-v(2\lambda_0-x_1,x''))+
(v(2\lambda_0-x_1,x'')-v(2\lambda-x_1,x''))\\&\geq &
\frac{\delta(\lambda_0-x_1)}{|x-\lambda_0e_1|^{n+2a}}-
\frac{c(\lambda_0-\lambda)}{|x-\lambda_0e_1|^{n+2a}}(\lambda_0-x_1+c)\\&=&
\frac{[\delta-c(\lambda_0-\lambda)]
(\lambda_0-x_1)-c^2(\lambda_0-\lambda)}{|x-\lambda_0e_1|^{n+2a}}>0,
\end{eqnarray*}
if $x_1\leq \lambda_0-\frac{\epsilon}{2}$ and $|\lambda-\lambda_0|$
is sufficiently small compared to $\epsilon$. This completes the proof of
the present lemma.
\end{proof}

 In all the above arguments, the discussion is always carried out outside a
 neighborhood of the origin. Now let us investigate the behavior of
 $v$ in this neighborhood. The following idea mainly comes from \cite{L2} and \cite{CL2}.

\begin{lemma}\label{lem005}
Suppose that $v\in C^2(B_1\backslash \{0\})\cap C(\bar{B}_1\backslash\{0\})$ is a  positive solution
to the following problem with $n+2a>2$,
\begin{equation}L(v)=\Delta_{n+1} v+\frac{2a-1}{x_{n+1}}\partial_{n+1}v\leq 0 \text{
in } B_1\backslash\{0\} \text{ with } \partial_{n+1}v(x',0)=0.
\end{equation} Then there holds $$v(x)\geq \underset{\partial B_1}{\inf}
v, \ \   \   \   \   \   \   \   \   \forall x\in B_1\backslash
\{0\}.$$
\end{lemma}
\begin{proof}
Set $\underset{\partial B_1}{\inf} v=m_1$. Consider $h_s(x)$, with $0<s<1$ and a suitable constant $l(s)$,
$$h_s(x)=m_1+l(s)(\frac 1{|x|^{n+2a-2}}-1),
h_s|_{\partial B_1}=m_1,h_s|_{\partial B_s}=-1.$$ From the definition
of $h_s(x)$, there holds$$L(v-h_s)\leq 0\text{ in } B_1\backslash
B_s,\ \ v-h_s|_{\partial B_1}\geq 0, \    \ v-h_s|_{\partial
B_s}\geq 1.$$ By the maximum principle we have for any fixed $x$
$$v(x)\geq h_s(x)=m_1+l(s)(\frac 1{|x|^{n+2a-2}}-1),\    \  \text{ if } s \text{ is small }.$$ It is
easy to see
$$l(s)=-\frac{(m_1+1)s^{n+2a-2}}{1-s^{n+2a-2}}\rightarrow 0 \text{
as } s\rightarrow 0.$$Thus  passing to the limit $s\rightarrow 0$, we
have proved $v(x)\geq m_1$. This finishes the proof of the present
lemma.
\end{proof}
\section{The proof for Theorem \ref{thm1}}
\label{sec:2}
Now we can  prove Theorem \ref{thm1}.
\par Proof of Theorem \ref{thm1}:
Set
$$\lambda_0=\underset{\lambda}{\inf}\{\lambda\big|v(x)>v(x^{\lambda'})\ \
\forall x\in \Sigma_{\lambda'},\lambda'\geq \lambda\}.$$ By Lemma
\ref{lem003} and Lemma \ref{lem005}, we have  $|\lambda_0|<\infty$. Also we can claim that
\begin{equation}\label{112}v(x)=v(x^{\lambda_0}),\ \ \forall x\in
\Sigma_{\lambda_0}\backslash\{0\}.\end{equation}

Without loss of generality, we may assume $\lambda_0> 0$. Since for the case $\lambda_0=0$, we can start the moving plane from $-\infty$  and stop at $x_1=\lambda_1$. If $\lambda_1<0$, we can prove $v(x)=v(x^{\lambda_1})$ by the
 same arguments as we do in the case $\lambda_0>0$. Otherwise $\lambda_1=0$, the claim \eqref{112} holds immediately. Now we turn to prove the claim \eqref{112} for $\lambda_0>0$. If \eqref{112} is false, by Lemma \ref{lem001}, one gets
$$v(x)>v(x^{\lambda_0}),\ \ \forall x\in
\Sigma_{\lambda_0}\backslash\{0\}.$$ Also from the definition of $\lambda_0$ and Lemma \ref{lem004}, one can choose $\lambda_k\uparrow \lambda_0$ as $k\rightarrow \infty$ such that $$
 \emptyset\neq\sigma_k=\{x\in \Sigma_{\lambda_k}\backslash\{0\}\big|v(x)
\leq v(x^{\lambda_k})\}\subset B_R.$$ Noting Lemma \ref{lem005} and the     continuity of $v$, we can see that $$
v(x)>v(x^{\lambda_k}) \text{ in } B_r\backslash\{0\},\  B_r\subset \Sigma_{
\lambda_k},$$ if we choose $r$ small enough and $k$ large enough. This implies that $\sigma_k\subset B_R\backslash B_r$. Set $w_k(x)=v(x)-v(x^{\lambda_k})$ and $w(x)=v(x)-v(x^{\lambda_0})$. It is easy to see that $\exists x^k\in \sigma_{
k}$ such that $w_k(x^k)=\inf w_k(x)$. Next we consider $x^\infty=\underset{k\rightarrow \infty }{\lim} x^k$ in two cases

(1) $x^\infty\in \Sigma_{\lambda_k}$: we have $w(x^\infty)=\underset{k\rightarrow \infty}{\lim} w_k(x^k)\leq 0$ which is a contradiction.

(2) $x^\infty \in \partial\Sigma_{\lambda_k}$: we have $\partial_{x_1}w(x^\infty)=\underset{k\rightarrow \infty}{\lim} \partial_{x_1}w_k(x^k)= 0$ which is a contradiction too.

This proves the assertion \eqref{112}.

If $\alpha<\frac{n+2a+2}{n+2a-2}$, we have $\tau>0$. To prove the radial symmetry of $v$, one should take a
transformation. Set
$$\bar{v}(x',x_{n+1},x_{n+2})=v(x',\sqrt{x_{n+1}^2+x_{n+2}^2}).$$
It follows that, \begin{equation}\label{301}
\Delta_{n+2}\bar{v}+\frac{2a-2}{x_{n+2}}\partial_{n+2}\bar{v}+
|x|^{-\tau}\bar{v}^\alpha=0,
\text{ in }R^{n+2},\partial_{n+2}\bar{v}(x',x_{n+1},0)=0.
\end{equation}
There is a
singularity at $0$, and hence $\lambda_0$ must be $0$. Notice that
 \eqref{301}  is rotation invariant about $x', x_{n+1}$.  We have
$$v(x',x_{n+1})=\bar{v}(x',x_{n+1},0)=\bar v(\bar{x}',\bar x_{n+1},0)=v(\bar x',\bar x_{n+1}),\text{ if }|x'|^2+x_{n+1}^2=|\bar{x}'|^2+\bar{x}_{n+1}^2.$$ This implies that $$\bar{u}(x',x_{n+1})=\bar u(\bar{x}',\bar x_{n+1}),\text{ if }|x'|^2+x_{n+1}^2=|\bar{x}'|^2+\bar{x}_{n+1}^2.$$
 If we take another transformation such
as
$$v_b(x)=\frac 1{|x|^{n+2a-2}}\bar{u}_b(\frac{x}{|x|^2}),\text{ here
}b_{n+1}=0,$$  where $\bar{u}_b(x)=\bar{u}(x-b)$. Repeating the above arguments, similarly we have
$$\bar u(x',x_{n+1})=\bar u(\bar{x}',\bar x_{n+1}), \text{ if }|x'+b'|^2+x_{n+1}^2=|\bar{x}'+b'|^2+\bar x_{n+1}^2.$$ In fact, $b'$ can
be chosen arbitrarily, thus $\bar u$ must be a constant. This means that $\bar u\equiv 0$.

Now we consider the case $\alpha=\frac{n+2a+2}{n+2a-2}$ or $\tau=0$.
By the same arguments as we did in the case $\tau>0$,
there exists $\lambda=(\lambda_1,...,\lambda_{n+1})$ such that
\begin{equation}\label{109}
\bar{v}(x',x_{n+1},0)=v(x',x_{n+1})=v(\bar{x}',\bar{x}_{n+1})
=\bar{v}(\bar{x}',\bar{x}_{n+1},0),\text{ if }
\sum_{i=1}^{n+1}|x_i-\lambda_i|^2=
\sum_{i=1}^{n+1}|\bar{x}'_i-\lambda_i|^2.
\end{equation} In fact, $\lambda_{n+1}$ must be 0. Otherwise, it follows that $$
v(x',2\lambda_{n+1}-x_{n+1})=v(x',x_{n+1})=v(x',-x_{n+1}).$$ It shows
that for the fixed $x'$, $v$ is periodic with respect to $x_{n+1}$
with period $2\lambda_{n+1}$. This means that $v$ must vanish which
is impossible. For $\lambda'=(\lambda_1,...,\lambda_n)$, we have two
cases.\begin{itemize}
\item[(1)] $\lambda'=0$:  noting
$\bar{u}(x)=\frac{1}{|x|^{n+2a-2}}v(\frac{x}{|x|^2})$, $\bar{u}(x)$
is radially symmetric with respect to the origin.
\item[(2)] $\lambda'\neq 0$: This means that $0$ is not the
symmetric center of $v$, $v$ must be $C^2$ at $0$. In other words,
$\bar{u}(x)$ has the similar asymptotic behavior at $\infty$ as
$v(x)$. This allows us to apply the  moving plane method to
$\bar{u}(x)$ directly to obtain that $\bar{u}(x)$ is radially
symmetric with respect to some point $b\in R^{n+1},b_{n+1}=0$.
\end{itemize}
The above arguments show  that $\bar u(x)$ is radially symmetric with
respect to a point $b\in \{b_{n+1}=0\}$. Now we can follow the
arguments of Section 3 in \cite{CLO}, then we can complete the proof of Theorem
\ref{thm1}.

Comparing Theorem \ref{thm1} with \eqref{004}, we can
regard \eqref{001} as an equation defined in dimension $n+2a$.
Therefore, we can consider the following more general equation
\begin{equation}\begin{cases}\label{003} \displaystyle
\sum_{i=1}^m y_iu_{y_iy_i}+ \sum_{i=1}^{m}a_i u_{y_i}+\Delta_x
u+u^\alpha=0\text{ in } R^{m,n}_+=\{(x,y)\big|x\in R^n,y_i\in R^1_+,i=1,...,m\},\\
u\geq 0 \text{ in } R^{m,n}_+ \text{ and } u\in
C^2(\overline{R^{m,n}_+}).\end{cases}
\end{equation}
\begin{theorem}\label{thm2}
Let $u(x,y)$ be a nonnegative solution of \eqref{003} with constants
$a_i>1,i=1,...,m$. Then with $a=\displaystyle \sum_{i=1}^m a_i$
\begin{itemize}
\item[(1)] for $1<\alpha<\frac{n+2a+2}{n+2a-2},
\   u\equiv 0$. \item[(2)] for $\alpha=\frac{n+2a+2}{n+2a-2}$,
$u_{t,x_0}(x,y)=
\displaystyle\left(\frac{t\sqrt{(n+2a)(n+2a-2)}}{t^2+4\sum_{i=1}^m
y_i+|x-x_0|^2}\right)^{\frac{n+2a-2}{2}}$,
\end{itemize}for some $x_0\in R^n$ and $t\geq 0$.
\end{theorem}
The proof of Theorem \ref{thm2} is just the same as the proof of
Theorem \ref{thm1}, as we can easily establish the similar lemmas as
in Section 2 for \eqref{003}. Thus we omit the details here.

\section{An application to a priori estimates
of semi-linear degenerate elliptic equations}
\label{sec:3}
The proof for Theorem \ref{thm501}: our proof is by contradiction and uses a scaling argument reminiscent to that used in the theory of Minimal Surfaces, also refer to   \cite{GS2}. If \eqref{888} is false, we
can get a sequence $u^k\in C^2(\Omega)\cap L^{\infty}(\Omega)$ such
that
\begin{equation} |u^k|_{L^\infty}=M_k\rightarrow \infty \text{ as }
k\rightarrow \infty.
\end{equation} Hence, we can find $x^k\in \Omega\rightarrow
\bar{x}\in \bar\Omega$ as $k\rightarrow \infty$ such that
$u^k(x^k)\geq \frac {M_k}2$. Next we shall distinguish two cases to investigate.

\underline{Case 1: $\bar x\in \Omega$}. With $y=\frac{x-x^k}{\mu_k}$ define the scaled function \begin{equation}
v^k(y)=\mu_k^{\frac{2}{\alpha-1}}u^k(x) \text{ where } \mu_k^{\frac{2}{\alpha-1}}M_k=1.
\end{equation} For large $k$, $v^k(y)$ is well defined in $B_{\frac{d}{\mu_k}}(0)$ where $2d=\text{dist}(\bar x,\partial \Omega)$, \begin{equation}
\underset{y\in B_{\frac{d}{\mu_k}}(0)}{\sup}v^k(y)=1,v^k(0)\geq \frac 12,
\end{equation} Moreover, $v^k(y)$ satisfies \begin{equation}\label{d01}
a^{ij}_k\frac{\partial^2v^k}{\partial y_i\partial y_j}+\mu_kb^i_k\frac{\partial v^k}{\partial y_i}+\mu_k^{\frac{2\alpha}{\alpha-1}}f(\mu_k y+x^k,\mu_k^{-\frac 2{\alpha-1}}v^k)=0,\text{ in } B_{\frac{d}{\mu_k}}(0)
\end{equation}
where $a^{ij}_k(y)=a^{ij}(\mu_k y+x^k),b^i_k(y)=b^i(\mu_k y+x^k)$. Noting that $y\in B_{\frac{d}{\mu_k}}(0)$ which implies $\mu_k y+x^k\in B_d(x^k)$ and $\text{dist}(B_d(x^k),\partial \Omega)>\frac d2$ for $k$ large enough, one has \eqref{d01} is uniformly elliptic in $B_{\frac{d}{\mu_k}}(0)$. From \eqref{604}, we see that $$\underset{k\rightarrow \infty}{\lim}|\mu_k^{\frac{2\alpha}{\alpha-1}}f(\mu_k y+x^k,\mu_k^{-\frac 2{\alpha-1}}v^k)-h(\mu_k y+x^k)(v^k(y))^\alpha|=0.$$ Therefore, given any  $R$ such that $B_R(0)\subset B_{\frac{d}{\mu_k}}(0)$, we can, by elliptic $L^p$ estimates, find uniform bounds for $\|v^k\|_{W^{2,p}(B_R(0))}$. Choosing $p$ large, we obtain by Sobolev embedding theorem that $\|v^k\|_{C^{1,\beta}(B_R(0))}$, $0<\beta<1$, is also uniformly bounded. Passing to the limit $k\rightarrow\infty $ gives $v^k\rightarrow v$ and $v$ solves \begin{equation}\label{d02}
a^{ij}(\bar x)\frac{\partial^2 v}{\partial y_i\partial y_j}+h(\bar x)v^\alpha=0,\text{ in } R^2, v(0)\geq \frac 12.
\end{equation}By performing a rotation and stretching of coordinates, \eqref{d02} is reduced to \begin{equation}\label{e01}
\Delta v+v^\alpha=0 \text{ in } R^2.
\end{equation}  Suppose $v$ is a non-trivial non-negative solution of \eqref{e01}. Let $\tilde v(y_1,y_2,y_3)=v(y_1,y_2)$. Then \begin{equation}
   \Delta \tilde v+\tilde v^\alpha=0\text{ in } R^3
   \end{equation}Noting $\alpha<\frac{3+2a}{2a-1}<\frac{3+2}{3-2}=5$ and the results of \cite{GS}, we must have $\tilde v\equiv 0$ which contradicts to $\tilde v(0,y_3)\geq \frac 12$.

\underline{Case 2: $\bar x \in \partial \Omega$}. This is quite different from Case 1.
Without loss of generality, we may
assume that $$\partial_1\phi(\bar{x})=0,\partial_2\phi(\bar{x})\neq 0.$$
 From \eqref{666} it follows $$0=a^{ij}\partial_i\phi\partial_j\phi(\bar{x})=
 a^{22}(\partial_2\phi)^2\Rightarrow a^{22}(\bar{x})=0.$$ Hence $a^{11}(\bar x)>0$ follows immediately from \eqref{603}.
 Denote$$
y_1=x_1, y_2=\phi(x), \    \   \forall x\in B_{d}(\bar{x})\cap
\bar{\Omega}, \ d\text{ small enough}.$$ Therefore, in the new
coordinates $(y_1,y_2)$, \eqref{601} can be written as for some small $\delta$ \begin{eqnarray} \label{605} \tilde
{a}^{22}\frac{\partial^2 u^k}{\partial y_2^2}+
\tilde{a}^{11}\frac{\partial^2 u^k}{\partial
y_1^2}+2\tilde{a}^{12}\frac{\partial^2 u^k}{\partial y_1\partial
y_{2}}+\tilde{b}^1\frac{\partial u^k}{\partial
y_1}+\tilde{b}^2\frac{\partial u^k}{\partial y_{2}}+f(y,u^k)=0,\text{ in }B_{\delta}(y^k)\cap \{y_2>0\}
\end{eqnarray}where \begin{equation}
\tilde{a}^{22}=a^{ij}\partial_i\phi\partial_j\phi,
\tilde{a}^{11}=a^{11},\tilde{a}^{12}=a^{1j}\partial_j\phi,
\tilde{b}^1=b^1,\tilde{b}^2=
b^j\partial_j\phi+a^{ij}\partial_{ij}\phi.
\end{equation}Set $$p_1=\frac{y_1-y_1^k}{\mu_k},
p_{2}=\frac{y_{2}-y_{2}^k}{\mu_k^2},
v^k(p)=\mu_k^{\frac{2}{\alpha-1}}u^k(y)\text{ with }\mu_k^{\frac{2}{\alpha-1}}M_k=1.$$
Then
\begin{eqnarray} \label{606}
\mu_k^{-2}\tilde a^{22}\frac{\partial^2 v^k}{\partial p_{2}^2}+
\tilde a^{11}\frac{\partial^2 v^k}{\partial
p_1^2}&+&2\mu_k^{-1}\tilde a^{12}\frac{\partial^2 v^k}{\partial
p_1\partial p_{2}}+\mu_k\tilde b^1\frac{\partial v^k}{\partial
p_1}\\&+&\tilde b^2\frac{\partial v^k}{\partial
p_{2}}+\mu_k^{\frac{2\alpha}{\alpha-1}}f(p,\mu_k^{-\frac{2}{\alpha-1}}v^k)=0\nonumber
\text{ in }B_{\frac{\delta}{\mu_k}}(0)\cap\{p_2>-\frac{y^k_2}{\mu^2_k}\}.\end{eqnarray}
Set $H_k=B_{\frac{\delta}{\mu_k}}(0)\cap\{p_2>-\frac{y^k_2}{\mu^2_k}\}$. Then we will have the
following lemma\begin{lemma}
\label{lem401}In the region considered, one has \begin{equation}\label{401}\tilde{a}^{11}\geq c_0>0,\tilde{a}^{12}=A^{12}(p)\mu_k^2(p_2+\frac{y_2^k}{\mu_k^2}),\tilde{a}^{22}=
A^{22}(p)\mu_k^2(p_2+\frac{y^k_2}{\mu_k^2}),\end{equation}where \begin{equation}\label{a111}A^{12},A^{22}\in C^1(\bar{H}_k),A^{22}(p_1,-\frac{y_2^k}{\mu^2_k})>0,
\frac{\tilde{b}^2(p_1,-\frac{y^k_2}{\mu_k^2})}
{A^{22}(p_1,-\frac{y_2^k}{\mu^2_k})}>2.\end{equation}
\end{lemma}\begin{proof}
Noting $\tilde{a}^{22}=a^{ij}\partial_i\phi\partial_j\phi=0$ on $\{y_2=0\}$, we get \begin{eqnarray}\tilde{a}^{22}(\bar y)&=&\int_0^1\frac{d(a^{ij}\partial_i\phi\partial_j\phi)(
\bar y_1,t\bar y_2)}{dt}dt=\bar y_2\int_0^{ 1}\partial_{y_2}(a^{ij}\partial_i\phi\partial_j\phi)(\bar y_1,t\bar y_2)dt \nonumber\\&=&A^{22}\mu^2_k(\bar p_2+\frac{y^k_2}{\mu_k^2}),\text{ where }A^{22}=\int_0^{1}\partial_{y_2}(a^{ij}\partial_i\phi\partial_j\phi)(\bar y_1,t\bar y_2)dt.\end{eqnarray}From \eqref{666}, we see that $\nabla(\tilde a^{22})\neq 0$ on $\{y_2=0\}$ which implies that $\partial_{y_2}\tilde a^{22}(y_1,0)>0$ or $A^{22}(p_1,-\frac{y_2^k}{\mu_k^2})>0$ in the region considered. The $C^1$ property of $A^{22}$ follows from the $C^2$ property of $a^{ij},\phi$ immediately. The last term in \eqref{a111} follows from \eqref{a101}.

\end{proof}

Dividing both sides of \eqref{606} by $A^{22}$, one can get in $H_k$ \begin{eqnarray}\label{402}
(p_2+\frac{y_2^k}{\mu_k^2})\frac{\partial^2 v^k}{\partial p_{2}^2}+
\bar a^{11}\frac{\partial^2 v^k}{\partial
p_1^2}&+&2\mu_k (p_2+\frac{y_2^k}{\mu_k^2})\bar a^{12}\frac{\partial^2 v^k}{\partial
p_1\partial p_{2}}\nonumber\\&+&\mu_k\bar b^1\frac{\partial v^k}{\partial
p_1}+\bar b^2\frac{\partial v^k}{\partial
p_{2}}+\mu_k^{\frac{2\alpha}{\alpha-1}}g(p,\mu_k^{-\frac{2}{\alpha-1}}v^k)=0.
\end{eqnarray}where $$\bar{a}^{11}=\displaystyle\frac{\tilde a^{11}}{A^{22}}, \bar a^{12}=\frac{A^{12}}{A^{22}},\bar b^i=\frac{\tilde b^i}{A^{22}},\bar f=\frac{f}{A^{22}}.$$ We must take care of the limit of $\frac{y_2^k}{\mu_k^2}$.

\underline{Case 2.1}: $\underset{k\rightarrow}{\lim}\frac{y_2^k}{\mu_k^2}=\infty$. We take $q_1=p_1,q_2=2\sqrt{p_2+\frac{y_2^k}{\mu_k^2}}-2\sqrt{\frac{y_2^k}{\mu_k^2}}$, then \eqref{402} changes to \begin{eqnarray}\label{d03}
\frac{\partial^2 v^k}{\partial q_{2}^2}&+&
\bar a^{11}\frac{\partial^2 v^k}{\partial
q_1^2}+ \mu_k \left(q_2+2\sqrt{\frac{y_2^k}{\mu_k^2}}\right)\bar a^{12}\frac{\partial^2 v^k}{\partial q_1\partial q_{2}}\nonumber\\&+&\mu_k\bar b^1\frac{\partial v^k}{\partial
q_1}+\frac{2\bar b^2-1}{q_2+2\sqrt{\frac{y_2^k}{\mu_k^2}}}\frac{\partial v^k}{\partial
q_{2}}+\mu_k^{\frac{2\alpha}{\alpha-1}}\bar f(q,\mu_k^{-\frac{2}{\alpha-1}}v^k)=0,\text{ in } J_k.
\end{eqnarray}
It is important to show that $J_k$ can be chosen arbitrarily large as $k\rightarrow \infty$. Since $p\in H_k$, it follows that \begin{eqnarray}
&&q_1^2+\left[\left(\frac{q_2}{2}+\sqrt{\frac{y_2^k}{\mu_k^2}}\right)^2-
\frac{y^k_2}{\mu_k^2}\right]^2<\frac{\delta^2}{\mu_k^2}\iff q_1^2+q_2^2\left(\frac{q_2}{4}+\sqrt{\frac{y_2^k}{\mu_k^2}}\right)^2<\frac{\delta^2}{\mu_k^2}
\nonumber\\&&\Rightarrow q_2<\frac{2\delta}{\sqrt{y_2^k}} \text{ noting that } q_2>-2\sqrt{\frac{y_2^k}{\mu_k^2}}.
\end{eqnarray}From $y_2^k\rightarrow 0$ and $\frac{y_2^k}{\mu_k^2}\rightarrow \infty$, one can get\begin{equation}
\frac{\delta^2}{\mu_k^2}=4l_k^2{\max} ^2\left\{\frac{\delta}{2\sqrt{y_2^k}},\sqrt{\frac{y_2^k}{\mu_k^2}}\right\},l_k\rightarrow
\infty\text{ as }k\rightarrow \infty.
\end{equation}Since \begin{eqnarray}
q_1^2+q_2^2\left(\frac{q_2}{4}+\sqrt{\frac{y_2^k}{\mu_k^2}}\right)^2\leq q_1^2+4{\max} ^2\left\{\frac{\delta}{2\sqrt{y_2^k}},\sqrt{\frac{y_2^k}{\mu_k^2}}\right\}q_2^2,
\end{eqnarray} we can take $J_k=B_{l_k}(0)\cap\{q_2>-2\sqrt{\frac{y_2^k}{\mu_k^2}}\}$.

Also, we have $v^k(0)\geq \frac 12$. As for any $R$, we can choose $k$ large enough such that $B_R(0)\subset J_k$. Thus \eqref{d03} is uniformly elliptic in $B_R(0)$ with uniformly bounded coefficients. This allows us to follow the same steps in Case 1. Namely, passing to limit $k\rightarrow\infty$, we have \begin{equation}\label{d04}
\frac{\partial^2v}{\partial q^2_2}+\bar{a}^{11}(\bar x)\frac{\partial^2v}{\partial q_1^2}+\frac{h(\bar x)}{\partial_{y_2}(a^{ij}\phi_i\phi_j)(\bar x)}v^\alpha=0\text{ in } R^2, v(0)\geq \frac 12,
\end{equation}if we notice that for $q\in B_R(0)$, \begin{equation}
|\mu_k(q_2+2\sqrt{\frac{y_2^k}{\mu_k^2}})|+|\frac{2\bar b^2-1}{q_2+2\sqrt{\frac{y_2^k}{\mu_k^2}}}|\leq \mu_kR+2\sqrt{y^k_2}+\frac{C\mu_k}{\sqrt{y^k_2}}\rightarrow 0\text{ as }k\rightarrow \infty.
\end{equation} Therefor, \eqref{d04} gives rise a contradiction.

\underline{Case 2.2}:  $\underset{k\rightarrow}{\lim}\frac{y_2^k}{\mu_k^2}=c<\infty$.

First of all, we  establish a lemma of weighted-$L^2$ estimates. Let $\psi\in C_c^{\infty}(R^{2})$ be
 a cutoff function with $\psi(p)=1 $ as $|p|\leq  1/2 $ and $\psi=0$ as $|p|\geq 1$. Set $\psi_r(p)=\psi(\frac pr)$.\begin{lemma}
\label{lem402} Suppose $u\in C^2(R^2_+)\cap L^\infty(R^2_+)$ solves \eqref{403},\begin{equation}
\label{403}p_2\frac{\partial^2 u}{\partial p_{2}^2}+
B^{11}\frac{\partial^2 u}{\partial
p_1^2}+2p_2B^{12}\frac{\partial^2 u}{\partial
p_1\partial p_{2}}+B^1\frac{\partial u}{\partial
p_1}+B^2\frac{\partial u}{\partial
p_{2}}+f=0\text{ in } R^2_+,
\end{equation}with $B^{ij},B^i\in C^1(\overline{R^2_+}),f\in L^\infty(R^2_+)$ and $B^{11}(p_1,0)\geq c_0>0$. Then for $r$ suitable small, we have \begin{equation}\label{b111}
\|p_2^{\frac 12}\psi_ru_{p_2}\|_{L^2}+\|\psi_ru_{p_1}\|_{L^2}\leq C(r,\|\psi_rB^{ij}\|_{C^1},\|\psi_rB^i\|_{C^1},\|\psi_rf\|_{L^\infty},
\|\psi_ru\|_{L^\infty}).
\end{equation}
\end{lemma}
\begin{proof}
Set $\eta_{\epsilon}(p_2)\in C^\infty(R^1_+)$ that\begin{equation}
\eta_{\epsilon}(p_2)=\begin{cases} 0,\ 0<p_2<\epsilon\\ 1, \
p_2>2\epsilon,
\end{cases}
\end{equation}with $|D^j\eta_\epsilon|\leq C_j\epsilon^{-j}$ for $p_2\in (\epsilon,2\epsilon)$.   Denote $\psi_{r,\epsilon}=\psi_r\eta_\epsilon$. Multiplying both sides of \eqref{403} by $\psi_{r,\epsilon}u$  and integrating by parts, we can get \begin{eqnarray}\label{404}&&
\int \psi_{r,\epsilon}p_2\left(\frac{\partial u}{\partial p_2}\right)^2+\int B^{11}\psi_{r,\epsilon}\left(\frac{\partial u}{\partial p_1}\right)^2\nonumber\\&=&
\int \left(\frac{\partial \psi_{r,\epsilon}}{\partial p_2}+\frac 12p_2\frac{\partial^2 \psi_{r,\epsilon}}{\partial p_2^2}-\frac 12\frac{\partial(B^2\psi_{r,\epsilon})}{\partial p_2}\right)u^2+\int\left(B^1\psi_{r,\epsilon}-\frac{\partial(B^{11}\psi_{r,\epsilon})}
{\partial p_1}\right)u\frac{\partial u}{\partial p_1}\nonumber\\&-&2\int p_2\frac{\partial(B^{12}\psi_{r,\epsilon})}{\partial p_1}u\frac{\partial u}{\partial p_2}-2\int p_2B^{12}\psi_{r,\epsilon}\frac{\partial u}{\partial p_1}\frac{\partial u}{\partial p_2}.
\end{eqnarray}Now we estimate the terms on the right side of \eqref{404}. The first term,\begin{eqnarray}
\left|\int \frac{\partial \psi_{r,\epsilon}}{\partial p_2}u^2\right|\leq \int |\partial_{p_2}\psi_r|\eta_\epsilon u^2+\int |\partial_{p_2}\eta_\epsilon| \psi_r u^2\leq C_1+C_2\int_\epsilon^{2\epsilon}\epsilon^{-1}dp_2\leq C,
\end{eqnarray}where $C$ is a constant only depending on the quantities in \eqref{b111}.  Also\begin{eqnarray}
\left|\int p_2\frac{\partial\psi_{r,\epsilon}}{\partial p_2^2}u^2\right|&\leq& \int p_2|\partial^2_{p_2}\psi_r|\eta_\epsilon u^2+2\int p_2|\partial_{p_2}\psi_r\partial_{p_2}\eta_\epsilon|u^2+\int p_2\psi_r|\partial^2_{p_2}\eta_\epsilon|u^2\nonumber\\&\leq&C_1+\epsilon^{-2}
\int_{\epsilon}^{2\epsilon}p_2dp_2\leq C
\end{eqnarray} The second term,\begin{equation}
\left|\int B^1\psi_{r,\epsilon}u\frac{\partial u}{\partial p_1}\right|\leq \delta\int \psi_{r,\epsilon}\left(\frac{\partial u}{\partial p_1}\right)^2+\frac 1{4\delta}\int (B^1)^2\psi_{r,\epsilon}u^2.
\end{equation}The last term,\begin{eqnarray}
\left|\int p_2B^{12}\psi_{r,\epsilon}\frac{\partial u}{\partial p_1}\frac{\partial u}{\partial p_2}\right|&\leq& \int p_2^{\frac 32}\psi_{r,\epsilon}|B^{12}|\left(\frac{\partial u}{\partial p_2}\right)^2+\int p_2^{\frac 12}\psi_{r,\epsilon}|B^{12}|\left(\frac{\partial u}{\partial p_1}\right)^2\nonumber\\&\leq& Cr^\frac 12\left(\int \psi_{r,\epsilon}p_2\left(\frac{\partial u}{\partial p_2}\right)^2+\int \psi_{r,\epsilon}\left(\frac{\partial u}{\partial p_1}\right)^2\right).
\end{eqnarray}Combining the above estimates and choosing suitable $r,\delta$, one can get \begin{equation}
\int \psi_{r,\epsilon}p_2\left(\frac{\partial u}{\partial p_2}\right)^2+\int \psi_{r,\epsilon}\left(\frac{\partial u}{\partial p_1}\right)^2\leq C
\end{equation}for  some constant $C$ independent of $\epsilon$. Passing the limit $\epsilon\rightarrow 0$, we have finished the proof of the present lemma.
\end{proof}

 Now we can complete the proof of Case 2.2. Replacing $p_2+\frac{y_2^k}{\mu^2_k}$ by $p_2$, still denote it by $p_2$. Then by Lemma \ref{lem401}, one can get $\bar b^2(p_1,0)\geq b>2$. All the coefficients of \eqref{402} are $C^1(\overline{R^2_+})$ with $\mu_k^{\frac{2\alpha}{\alpha-1}}\bar f(p,\mu_k^{-\frac{2}{\alpha-1}}v^k)\in L^\infty$, this means all the requirements in Lemma \ref{lem402} are fulfilled.
From Lemma \ref{lem401}, Lemma \ref{lem402} and the regularity results of Theorem \ref{thm601} and the standard regularity results for non-degenerate elliptic equations, one can choose a suitable subsequence
such that $v^k(0',-\frac{y_2^k}{\mu_k^2})\rightarrow v(0',c)\geq \frac 12$ and also
$v^k\rightarrow v$ in the distribution sense in $\mathscr{D}'(R^2_+)$ and $v$
satisfies\begin{equation}
(p_2+c)\partial_{y_2}(a^{ij}\phi_i\phi_j)(\bar x)\frac{\partial^2
v}{\partial p_2^2}+a^{11}(\bar x)\frac{\partial^2 v}{\partial
p_1^2}+(b^i\phi_i+a^{ij}\partial_{ij}\phi)(\bar x)\frac{\partial v}{\partial
p_2}+h(\bar x)v^\alpha=0, \text{ in }
R^2_+,
\end{equation}
where $0\leq c=\underset{k\rightarrow
\infty}{\lim}\frac{y^k_2}{\mu_k^2}<\infty$. By a linear change of coordinates and a stretching of
coordinates, we have that\begin{equation}
\begin{cases}
p_2v_{p_2p_2}+v_{p_1p_1}+\bar{b}v_{p_2}+v^\alpha=0 \text{ in }R^2_+,\\\
0\leq v\in C^2(R^2_+)\cap C(\overline{R^2_+}), v(0,c)=c_0>0.
\end{cases}
\end{equation} From the assumption of Theorem \ref{thm501}, it
follows that $2<\bar{b}\leq a$ and $$
\alpha<\frac{2a+3}{2a-1}\leq \frac{2\bar b+3}{2\bar b-1}.$$
By Theorem \ref{thm601}, we see that $v\in C^2(\overline{R^2_+})$ and have
$v\equiv 0$ which follows from Theorem \ref{thm1}. This is a
contradiction to $v(0,c)>0$. This ends the proof of Theorem \ref{thm501}.

\section{Appendix}
\label{sec:4}
In the present Appendix, we shall give a result about the regularity of solutions to some degenerate elliptic equation in \cite{HH}. For the convenience of readers, we shall give a brief proof for it.
We shall use the notations in \cite{HH}.
Define $I_q(v)$ and $I_\beta(v)$ by:\begin{equation}
I_q(v)=\|y\partial_{yy}v\|_{L^q(R^{n+1}_+)}+\|\Lambda_1^2v\|_{L^q(R^{n+1}_+)}+
\|y^{\frac{1}{2}}\Lambda_1
v_y\|_{L^q(R^{n+1}_+)}+\|v_y\|_{L^q(R^{n+1}_+)}+\|v\|_{L^q(R^{n+1}_+)},
\end{equation}\begin{equation}
I_\beta(v)=[y\partial_{yy}v]_{\dot{C}^\beta(\overline{R^{n+1}_+})}+[
\Lambda_1^2v]_{\dot{C}^\beta(\overline{R^{n+1}_+})}
+[y^{\frac{1}{2}}\Lambda_1
v_y]_{\dot{C}^\beta(\overline{R^{n+1}_+})}+[v_y]_{\dot{C}^\beta(\overline{R^{n+1}_+})}+
\|v\|_{L^\infty(R^{n+1}_+)},
\end{equation}where $\Lambda_1$ is a singular integral operator with the symbol
$\sigma(\Lambda_1)=|\xi|$. Also we say a function $v(x, y)$ in $\dot{C}^\alpha( \overline{
R^{n+1}_+})$, $\alpha\in R_+^1\backslash
Z$, if
\begin{equation}|v|_{\dot{C}^\alpha(\overline{R^{n+1}_+})}=
\sum_{|\beta|\le
[\alpha]}|D^{\beta}v|_{C(\overline{R^{n+1}_+})}+
[v]_{\dot{C}^{\alpha}(\overline{R^{n+1}_+})}<\infty,
\end{equation}
where \begin{equation}[v]_{\dot{C}^\alpha(\overline{R^{n+1}_+})}
=\sum_{|\beta|=[\alpha]}\underset{
y\ge 0, x\neq\bar{x}\in R^n}{sup}\Big(\frac{| D_x^\beta v(x, y)-D_x^\beta
v(\bar{x}, y)|}{|x-\bar{x}|^{\alpha}}\Big).\end{equation}
 Let $\psi\in C_c^{\infty}(R^{n+1})$ be
 a cutoff function with $\psi(x,y)=1 $ as $|x|\le 1/2 $, $y\le 1/2 $ and $\psi=0$ as $|x|\ge 1$
 or $|y|\ge 1$. Set $\psi_r(x,y)=\psi(\frac xr,\frac yr)$.

\begin{lemma}\label{lem602}(Lemma 5.4 in \cite{HH}) Suppose that $u\in
C^2(R_+^{n+1})\bigcap L^p(R_+^{n+1})$ with $u_x,
yu_y\in L^p(R_+^{n+1})$ satisfies
\begin{equation}\begin{split}\label{611}
L(u)=yu_{yy}+\sum_{i,
j}a_{ij}u_{x_ix_j}+y\sum_ja_ju_{yx_j}+\sum_jb_ju_{x_j}+bu_y=f,
\text{in} \   R_+^{n+1}, \end{split}
\end{equation}where $a_{ij}, a_j, b_j, b$ are all in
$C(\overline{R_+^{n+1}})$ with $a_{ij}(0)=\delta_{ij},
b(0)>\frac 32$, $f\in L^\infty(R^{n+1}_+)$ and that for some
$\epsilon>0$,
\begin{equation}\label{612}
\lim_{y\rightarrow 0}y^{b(0)-1-\epsilon}u(x, y)=0\text { uniformly
for all }x\in R^n.
\end{equation}
Then for sufficiently large $p$, there are $r=r(p)>0$ such that
\begin{equation} I_p(\psi_ru)\leq C_r,
\end{equation}for some constant depending only on
$p, \|\psi_{2r}f\|_{L^p}, \|\psi_{2r}u\|_{L^p},
\|\psi_{2r}u_x\|_{L^p}$ and $\|y\psi_{2r}u_y\|_{L^p}$ provided that
$p>n+1$ or $p>\frac {n+1}2 $ and $b(0)-2-\epsilon>0$.
\end{lemma}
\begin{lemma}\label{lem603}(Lemma 5.5 in \cite{HH})
Suppose that $w, \partial_xw,y\partial_yw \in\dot
{C}^{\alpha}_{loc}(\overline{R_+^{n+1}})\cap C^2(R_+^{n+1})$ with
$\alpha\in R_{+}^1 \backslash Z$ and $w$ satisfies
(\ref{611}),where $a_{ij}, a_j, b_j, b, f$ are all in
$\dot{C}^\alpha_{loc}(\overline{R_+^{n+1}})$ with
$a_{ij}(0)=\delta_{ij}, b(0)>\frac 32$.  Then
\begin{equation}\label{613}
I_{\alpha}(\psi_rw)\le C,
\end{equation}
for some positive constants $r$ and $C$,  depending on  $\alpha,
|\psi_{2r}f|_{\alpha}, |\psi_{2r}w |_{\alpha}$,  $
|\psi_{2r}\partial_xw |_{\alpha}$ and
$|y\psi_{2r}\partial_yw|_{\alpha}$.
\end{lemma}

Denote by $ W^{1,p}_{\alpha}(U)$ the completion of the space of all
the functions $u$ in $C^1(\bar U)$  under the norm
$$\Big(\int_{U}y^{p\alpha}|Du|^pdxdy+\int_{U}y^{p\alpha}|u|^pdxdy\Big)^{\frac 1{p}}.$$
Here we always assume $U\subset R_+^{n+1}$,  bounded and $\partial
U\cap \{y=0\}$ nonempty.
\begin{lemma}\label{lem604}(Lemma 8.3 in \cite{HH} Appendix B)
Let $U\in C^1$ be bounded domain and let $\alpha\in (0, 1)$. Then the following maps are continuous
\begin{eqnarray}
W^{1, p}_{\alpha}(U)&\hookrightarrow &C^{\beta}(\bar U)\text { where
}\beta=1-\alpha-\frac {n+1}p,  \text { if }
p>\frac {n+1}{(1-\alpha)},\label{A61}\\
W^{1, p}_{\alpha}(U)&\hookrightarrow &L^q(U)\text { where }q<\frac
{(n+1)p}{n+1-(1-\alpha)p}, \text { if } \frac1{1-\alpha}<p<\frac
{n+1}{(1-\alpha)}.\label{A71}
\end{eqnarray}
Moreover, for $p=2$ and $\forall \alpha\in (0,1)$, one can have
\begin{equation}\label{A72}
W^{1, 2}_{\alpha}(U)\hookrightarrow L^q(U)\text { where }q<\frac
{q_1}{1+2\alpha}\text { and }q_1=2+\frac 4n.
\end{equation}

\end{lemma}
With the above three lemmas, we can establish the following theorem concerning the regularity of solutions to  degenerate elliptic equation \eqref{611} for $n=1$ and $p=2$.
\begin{theorem}\label{thm601} Suppose that $b(0)>2$ and $u\in
C^2(R_+^{2})\bigcap L^\infty(R_+^{2})$ with $u_x,
yu_y\in L^2(R_+^{2})$ satisfies \eqref{611}. Then \begin{itemize}
\item[(1)] Suppose $a_{ij}, a_j, b_j, b\in C(\overline{R_+^{2}})$ and $f\in L^\infty(R^{2}_+)$. Then there exist two constants $r>0$ and $\beta\in (0,1)$ such that \begin{equation}\label{501}
    \|\psi_ru\|_{C^\beta(\overline{R^{2}_+})}+ \|\psi_ru_x\|_{C^\beta(\overline{R^{2}_+})}+ \|y\psi_ru_y\|_{C^\beta(\overline{R^{2}_+})}\leq C_r,
    \end{equation}for some constant depending only on
$\|\psi_{2r}f\|_{L^\infty}, \|\psi_{2r}u\|_{L^\infty},
\|\psi_{2r}u_x\|_{L^2}$ and $\|y\psi_{2r}u_y\|_{L^2}$.
\item[(2)] Suppose $a_{ij}, a_j, b_j, b,f\in \dot{C}^{k+\beta}(\overline{R_+^{2}}),k\geq 0$. Then there  exist two constants $r>0$ and $\beta\in (0,1)$  such that \begin{equation}\label{502}
    I_{k+\beta}(\psi_ru)\leq C_r,
    \end{equation}for some constant depending only on
$\|\psi_{2r}f\|_{C^k}, \|\psi_{2r}u\|_{L^\infty},
\|\psi_{2r}u_x\|_{L^2}$, $\|y\psi_{2r}u_y\|_{L^2}$ and the $C^k$-norm of the coefficients.
\end{itemize}
\end{theorem}\begin{proof}
We first prove \eqref{501}. By Lemma \ref{lem602}, one can get $$
I_2(\psi_ru)\leq C_r,i.e.,y\psi_ru\in H^2(R^2_+),y^{\frac 12}\psi_ru_{xy}\in L^2(R^2_+).$$Hence by Sobolev embedding theorem, it follows that $y\psi_ru_y\in L^p(R^2_+),\forall p\in [2,\infty)$.
By noting $\psi_ru_x\in W^{1,2}_{\frac 12}(R^2_+) $ and \eqref{A72}(where $n=1$), we can see that $\psi_ru_x\in L^{p_1}(R^2_+),\forall p_1\in [2,3)$.
Now we can apply Lemma \ref{lem602} again for $p_1$ and another smaller $r_1$(for simplicity  we always denote it by $r$) to get $$
I_{p_1}(\psi_{r}u)\leq C_{r},i.e., \psi_{r}u_x\in W^{1,p_1}_{\frac 12}(R^2_+),\psi_{r}u_x,\psi_{r}u_y\in L^{p_1}(R^2_+).$$
This implies that $\psi_{r}u\in W^{1,p_1}(R^2_+)$. Then $\psi_{r}u\in C^{\frac 15}(\overline{R^{2}_+})$ if we take $p_1=\frac 52$.  Using Lemma \ref{lem604} again, one can get $\psi_{r}u_x\in L^{p_2}(R^2_+),\forall p_2\in [2,12)$.
Again, by Lemma \ref{lem602}, one can get$$I_{p_2}(\psi_{r}u)\leq C_{r}, i.e.,\psi_{r}u_x\in W^{1,p_2}_{\frac 12}(R^2_+).$$ By Lemma \ref{lem604}, we can get\begin{equation}\label{aaa}\|\psi_{r_2}u\|_{C^{\frac 15}(\overline{R^2_+})}+\|\psi_{r}u_x\|_{C^{\frac 15}(\overline{R^2_+})}+\|y\psi_{r}u_y\|_{C^{\frac 15}(\overline{R^2_+})}\leq C_{r},\text{ if } p_2=\frac{20}{3}.\end{equation}This proves \eqref{501}.
Now we can prove \eqref{502} by induction on $k$. For $k=0$,
\eqref{aaa} means we can apply Lemma 5.2 to get
$$\|\psi_{r}yu_{yy}\|_{\dot C^{\frac 15}(\overline{R^2_+})}+\|\psi_{r}u_{xx}\|_{\dot C^{\frac 15}(\overline{R^2_+})}+\|y^{\frac 12}\psi_{r}u_{xy}\|_{\dot C^{\frac 15}(\overline{R^2_+})}+\|\psi_{r}u_y\|_{\dot C^{\frac 15}(\overline{R^2_+})}\leq C_{r}$$  For $k=1$, as $\psi_{r}\partial_x{u_x},y^{\frac 12}\psi_{r}\partial_y{u_x}\in C^{\frac 15}(\overline{R^2_+})$,  we can continue to apply Lemma 5.2 to $\psi_r u_x$ again to get $I_{\frac 15}(\psi_ru_x)\leq C_r$ namely,
$$\|\psi_{r}y\partial_{yy}(u_x)\|_{\dot C^{\frac 15}(\overline{R^2_+})}+\|\psi_{r}\partial_{xx}(u_{x})\|_{\dot C^{\frac 15}(\overline{R^2_+})}+\|\psi_{r}\partial_y(u_x)\|_{\dot C^{\frac 15}(\overline{R^2_+})}\leq C_{r}$$
This means $\psi_r \partial_x(u_{y})\in \dot C^{\frac 15}(\overline{R^2_+})$. Combining with $\psi_{r}y\partial_y(u_{y})\in \dot C^{\frac 15}(\overline{R^2_+})$ and applying Lemma 5.2 to $\psi_r u_y$ again, we can see that $I_{\frac 15}(\psi_ru_y)\leq C_r$, namely,$$\|\psi_{r}y\partial_{yy}(u_y)\|_{\dot C^{\frac 15}(\overline{R^2_+})}+\|\psi_{r}\partial_{xx}(u_{y})\|_{\dot C^{\frac 15}(\overline{R^2_+})}+\|\psi_{r}\partial_y(u_y)\|_{\dot C^{\frac 15}(\overline{R^2_+})}\leq C_{r}
$$ Also this implies that $u_{yy}\in C(\overline{R^2_+})$. For general $k$, repeat the above steps, we can get \eqref{502}.
\end{proof}

 {\section*{{\normalsize \bf  Acknowledgement}}} This work is  supported by a training program for
innovative talents of key disciplines, Fudan University. The author would like to thank the valuable suggestions of Professor J.X.Hong and Professor C.M.Li.





\end{document}